\documentclass[12pt,twoside]{amsart}
\usepackage{amssymb}
\usepackage{graphicx}
\usepackage[margin=1in]{geometry}

\newcommand{\beas}{\begin{eqnarray*}}
\newcommand{\eeas}{\end{eqnarray*}}
\newcommand{\bea}{\begin{eqnarray}}
\newcommand{\eea}{\end{eqnarray}}
\newcommand{\beq}{\begin{equation}}
\newcommand{\eeq}{\end{equation}}
\newcommand{\ben}{\begin{enumerate}}
\newcommand{\een}{\end{enumerate}}

\newtheorem{theorem}{Theorem}[section]
\newtheorem{lemma}[theorem]{Lemma}
\newtheorem{proposition}[theorem]{Proposition}

\newtheorem{conjecture}[theorem]{Conjecture}

\theoremstyle{definition}
\usepackage{latexsym}
\usepackage{color,hyperref}
\definecolor{darkblue}{rgb}{0,0,0.6}
\hypersetup{colorlinks,breaklinks,linkcolor=darkblue,urlcolor=darkblue,anchorcolor=darkblue,citecolor=darkblue}

\author[F. Zanello]{Fabrizio Zanello}
\address{Department of Mathematical Sciences, Michigan Tech, Houghton, MI 49931}
\email{zanello@mtu.edu}

\title[On Bergeron's positivity problem for $q$-binomial coefficients]{On Bergeron's positivity problem\\for $q$-binomial coefficients}
\begin{document}

\begin{abstract} F. Bergeron recently asked the intriguing question whether $\binom{b+c}{b}_q -\binom{a+d}{d}_q$  has nonnegative coefficients as a polynomial in $q$, whenever $a,b,c,d$ are positive integers, $a$ is the smallest, and $ad=bc$. We conjecture that, in fact, this polynomial is also always unimodal, and combinatorially show our conjecture for $a\le 3$ and any $b,c\ge 4$. The main ingredient will be a novel (and rather technical) application of Zeilberger's KOH theorem.
\end{abstract}


\maketitle

\section{Introduction} 

An interesting problem recently posed by F. Bergeron \cite{Ber}, which naturally arose in his studies of the $q$-Foulkes conjecture, is whether the coefficients of the symmetric polynomial $\binom{b+c}{b}_q -\binom{a+d}{d}_q$ always form a nonnegative sequence, for any choice of positive integers $a,b,c,d$ where $a$ is the smallest and $ad=bc$. Here, $\binom {m+n}{m}_q$ as usual denotes the \emph{$q$-binomial coefficient}
$$\frac{(1-q)(1-q^2)\cdots (1-q^{m+n})}{(1-q)(1-q^2)\cdots (1-q^m)\cdot (1-q)(1-q^2)\cdots (1-q^{n})}.$$
It is easily seen that $\binom {m+n}{m}_q$  is a symmetric polynomial in $q$ of degree $mn$.

We remark here that a special case of Bergeron's question already appeared in Abdesselam-Chipalkatti \cite{AC2}, as a consequence of a more general conjecture. The same authors proved in \cite{AC1} the case $a=2$ of their conjecture. (We thank A. Abdesselam for pointing this out to us.) 

In this note, we conjecture that not only nonnegativity but also unimodality does hold in this context, and provide a combinatorial proof of our conjecture for $a\le 3$ and any $b,c\ge 4$. Recall that a sequence of numbers is \emph{unimodal} if it does not increase strictly after a strict decrease. We have:
\begin{conjecture}\label{conj}
Fix any  positive integers $a,b,c,d$ such that $a$ is the smallest and $ad=bc$. Then the coefficients of the symmetric polynomial
$$\binom{b+c}{b}_q -\binom{a+d}{d}_q$$
are nonnegative and unimodal.
\end{conjecture}

Notice that symmetry is clear, since both $\binom{b+c}{b}_q$ and $\binom{a+d}{d}_q$ are symmetric polynomials of the same degree, $bc=ad$. Also note that the case $a=1$ of the conjecture is trivial, thanks to the unimodality of $\binom{b+c}{b}_q$ (see e.g. \cite{OH,Pr1,St5,Sy,Zei}). In this note, first we apply a recent result of Pak-Panova \cite{PP}, that we reproved combinatorially in \cite{Za}, to provide a short proof of Conjecture \ref{conj} when $a=2$, for any values of $b$ and $c$. Then, we employ Zeilberger's KOH theorem \cite{Zei} to settle the case $a=3$.

Unimodality results for suitable differences of two $q$-binomial coefficients have also appeared, for instance, in the work of Reiner and Stanton (\cite{RS}, Theorems 1 and 5), who employed interesting methods of representation theory in their proofs. See also their Conjecture 9 (or \cite{Stanton}, Conjecture 7), which provides a broad family of possible further identities, and is still wide open to this day. It would be interesting to investigate whether our combinatorial approach via the KOH theorem, which appears to be new in this context, might also help with the Reiner-Stanton conjecture.

\section{Proof of the conjecture for $a\le 3$}

We begin with a proof of Conjecture \ref{conj} for $a=2$. Notice that even this ``simplest'' case, which we are now able to show rather easily, heavily relies on the \emph{strict unimodality} of a $q$-binomial coefficient $\binom{b+c}{b}_q$ (i.e.,  $\binom{b+c}{b}_q$ is unimodal and its coefficients are strictly increasing through degree $\lfloor bc/2 \rfloor$, except from degree $0$ to $1$). This fact was established only recently.

\begin{lemma}[\cite{PP,Za}]\label{strict}
Assume $c\ge b\ge 2$. Then $\binom{b+c}{b}_q$ is strictly unimodal if and only if $b=c=2$ or $b\ge 5$, with the following nine exceptions:
$$(b,c)=(5,6), (5,10), (5,14), (6,6), (6,7), (6,9), (6,11), (6,13), (7,10).$$
\end{lemma}

\begin{proposition}
Conjecture \ref{conj} is true when $a=2$.
\end{proposition}

\begin{proof}
Let $$\binom{d+2}{2}_q=\sum_{0\le i\le 2d}a_iq^i {\ }{\ }{\ }  \text{and} {\ }{\ }{\ } \binom{b+c}{b}_q=\sum_{0\le i\le bc}b_iq^i,$$
where we can assume that $c\ge b\ge 3$. Proving the result is tantamount to showing that $b_i-b_{i-1}\ge a_i-a_{i-1}$, for all $i\le d=bc/2$. 

It is easy to see directly that, for $i\le d$, $a_i=\lceil (i+1)/2\rceil$. Hence, $a_i-a_{i-1}=1$ if $i$ is even, and $a_i-a_{i-1}=0$ if $i$ is odd. Thus, it suffices to show that $\binom{b+c}{b}_q$ is strictly increasing in all even degrees up to $bc/2$. 

If $b\ge 5$, the result  follows from Lemma \ref{strict}, since one can easily verify computationally that all nine exceptional $q$-binomial coefficients of the lemma satisfy the inequality $b_{i-1}<b_i$ when $i\le bc/2$ is even.

If $b\le 4$, the conclusion can be obtained in a few different ways.  For a combinatorial proof, one can rely on the fact that in a symmetric chain decomposition of the poset $L(b,c)$ for $b=3,4$ (we can assume $c$ even when $b=3$), at least one new chain is introduced in every even degree $\le bc/2$ (see \cite{lind,west}). When $b=4$, notice that the result also immediately follows from \cite{SZ}, Lemma 2.1 (b).
\end{proof}

Before proving Conjecture \ref{conj} for $a=3$, we recall Zeilberger's {KOH theorem} \cite{Zei}. This result rephrases, in algebraic terms, O'Hara's celebrated combinatorial proof of the unimodality of $q$-binomial coefficients \cite{OH}, by decomposing these latter into suitable finite sums of unimodal, symmetric polynomials with nonnegative coefficients. %all symmetric about $mn/2$.  

Fix  positive integers $m$ and $n$, and for any  partition $\lambda=(\lambda_1,\lambda_2,\dots)\vdash m$, set $Y_i= \sum_{1\le j\le i}\lambda_j$ for $i\geq 1$, and $Y_0=0$.

\begin{lemma} [\cite{Zei}]\label{koh}
We have $\binom{m+n}{m}_q=\sum_{\lambda\vdash m}F_{\lambda}(q)$, where
$$F_{\lambda}(q)= q^{2\sum_{i\geq 1}\binom{\lambda_i}{2}} \prod_{j\geq 1} \binom{j(n+2)-Y_{j-1}-Y_{j+1}}{\lambda_j-\lambda_{j+1}}_q.$$
\end{lemma}

\begin{theorem}
Conjecture \ref{conj} is true when $a=3$.
\end{theorem}

\begin{proof}
Let
$$\binom{d+3}{3}_q=\sum_{0\le i\le 3d}a_iq^i {\ }{\ }{\ }  \text{and} {\ }{\ }{\ }\binom{b+c}{b}_q=\sum_{0\le i\le bc}b_iq^i,$$
where $b,c\ge 4$. We can assume that $b\equiv 0$ (mod 3). With some abuse of notation, define the first difference of a $q$-binomial coefficient as its truncation in the middle degree; i.e., we denote by $(1-q) \binom{d+3}{3}_q$ the polynomial $ 1+\sum_{1\le i\le 3d/2}(a_i-a_{i-1})q^i$, and similarly, $(1-q) \binom{b+c}{b}_q=1+ \sum_{1\le i\le bc/2}(b_i-b_{i-1})q^i$. 

Thus, showing the result is equivalent to proving that $$(1-q) \binom{d+3}{3}_q\le (1-q) \binom{b+c}{b}_q,$$
where inequalities between polynomials are defined degreewise (i.e., by $\sum \alpha_iq^i \le \sum \beta_iq^i$  we mean $\alpha_i \le \beta_i$ for all $i$).

By Lemma \ref{koh}, we can decompose $\binom{d+3}{d}_q$ as
$$\binom{d+3}{3}_q= q^6\binom{d-1}{3}_q+q^2\binom{d-1}{1}_q\binom{2d-1}{1}_q+\binom{3d+1}{1}_q,$$
where the first term on the right side corresponds to the partition $(3)$ of 3, the second to $(2,1)$, and the third to $(1,1,1)$.

By iterating Lemma \ref{koh} a total of $b/3$  times on the right side, standard computations give us that $\binom{d+3}{3}_q$ equals:
\begin{equation}\label{a}
q^{2b}\binom{d-4b/3+3}{3}_q + \sum_{0\le i\le \frac{b-3}{3}}q^{6i+2}\binom{d-4i-1}{1}_q\binom{2d-8i-1}{1}_q + q^{6i}\binom{3d-12i+1}{1}_q.
\end{equation}

If we now consider the partition $(3,3,\dots,3)$ of $b$, we see that its contribution to the KOH decomposition of $\binom{b+c}{b}_q$, by Lemma \ref{koh}, is given by:
$$q^{\frac{b}{3}\cdot 2\cdot \binom{3}{2}}\binom{(c+2)b/3 - (b-3)-b}{3}_q=q^{2b}\binom{d-4b/3+3}{3}_q,$$
which is precisely the first summand in (\ref{a}).

Thus, by (\ref{a}), we want to show that
\begin{equation}\label{a1}
(1-q)\left(\sum_{0\le i\le \frac{b-3}{3}}q^{6i+2}\binom{d-4i-1}{1}_q\binom{2d-8i-1}{1}_q + q^{6i}\binom{3d-12i+1}{1}_q\right)\leq (1-q)\sum F_{\lambda}(q),
\end{equation}
where the sum on the right is as in the KOH decomposition of $\binom{b+c}{b}_q$, and is indexed over all partitions $\lambda \vdash b$, $\lambda \neq (3,3,\dots,3)$.

We have:
$$(1-q)\left(\sum_{0\le i\le \frac{b-3}{3}}q^{6i+2}\binom{d-4i-1}{1}_q\binom{2d-8i-1}{1}_q + q^{6i}\binom{3d-12i+1}{1}_q\right)$$$$=  \sum_{0\le i\le \frac{b-3}{3}}(1-q)q^{6i+2}(1+q+\dots+q^{d-4i-2})\cdot \frac{1-q^{2d-8i-1}}{1-q}+(1-q)q^{6i}\cdot \frac{1-q^{3d-12i+1}}{1-q} $$$$=\sum_{0\le i\le \frac{b-3}{3}}(1+q+\dots+q^{d-4i-2})(q^{6i+2}-q^{2d-2i+1})+(q^{6i}-q^{3d-6i+1}).$$

Note that both $2d-2i+1$ and $3d-6i+1$ are larger than $bc/2=3d/2$ for all of our indices $i$, since $c\ge 4$. Thus, since the first difference is defined to be up to degree $bc/2=3d/2$, the last displayed formula becomes:
\begin{equation}\label{aa}
\sum_{0\le i\le \frac{b-3}{3}}( q^{6i}+q^{6i+2}+q^{6i+3}+\dots+q^{d+2i}).
\end{equation}

It is the polynomial in (\ref{aa}) that we will bound with the first difference of suitable families of terms appearing in the KOH decomposition of $\binom{b+c}{b}_q$. We begin by dominating
$$\sum_{1\le i\le \frac{b-3}{3}}(q^{6i+2}+q^{6i+3}+\dots+q^{d+2i}).$$
Hence, for now, we are not concerned with $(q^2+q^3+\dots+q^{d})+\sum_{0\le i\le \frac{b-3}{3}}q^{6i}$. 

Consider the following partitions of $b$:
$$\lambda^{i,j}=(\lambda_1^{i,j}=3,\dots,\lambda_i^{i,j}=3,\lambda_{i+1}^{i,j}=2,\dots,\lambda_{i+j}^{i,j}=2,\lambda_{i+j+1}^{i,j}=1,\dots,\lambda_{b-2i-j}^{i,j}=1),$$
for any indices
$$1\le i\le (b-3)/3 {\ }{\ }{\ }  \text{and} {\ }{\ }{\ } 1\le j\le \lfloor b/2-2i(c-1)/c \rfloor .$$
Since $b\ge 6$ (because $ b\equiv 0$ (mod 3)) and
$$b/2-2i(c-1)/c\le (b-3i)/2$$
for $c\ge 4$, all  partitions $\lambda^{i,j}$ contain at least one part of each size 1, 2, and 3.

We want to show that
$$\sum_{i,j}(1-q)F_{\lambda^{i,j}}(q)\ge \sum_{1\le i\le \frac{b-3}{3}}(q^{6i+2}+q^{6i+3}+\dots+q^{d+2i}).$$
Employing Lemma \ref{koh}, the contribution of $\lambda^{i,j}$ to the KOH decomposition of $\binom{b+c}{b}_q$ is
$$F_{\lambda^{i,j}}(q)=$$$$q^{6i+2j}\binom{(c+2)i-6i+1}{1}_q\binom{(c+2)(i+j)-6i-4j+1}{1}_q\binom{(c+2)(b-2i-j)-2b+1}{1}_q$$$$=q^{6i+2j}\binom{ci-4i+1}{1}_q\binom{ci-4i+cj-2j+1}{1}_q\binom{bc-2ci-4i-cj-2j+1}{1}_q.$$

For all $i$ and $j$ as above, we have
$$(6i+2j)+(bc-2ci-4i-cj-2j+1)=bc-2ci+2i-cj+1>bc/2.$$
Therefore,
$$(1-q)F_{\lambda^{i,j}}(q)=(1-q)\binom{ci-4i+1}{1}_q\binom{ci-4i+cj-2j+1}{1}_q\cdot \frac{ q^{6i+2j}-q^{bc-2ci+2i-cj+1}}{1-q}$$$$=q^{6i+2j}\binom{ci-4i+1}{1}_q\binom{ci-4i+cj-2j+1}{1}_q$$$$=(1+q+\dots+q^{ci-4i})(q^{6i+2j}+q^{6i+2j+1}+\dots+q^{ci+2i+cj}).$$

Thus, our goal is to show that
\begin{equation}\label{1}
\sum_{i,j}(1+q+\dots+q^{ci-4i})(q^{6i+2j}+\dots+q^{ci+2i+cj})\ge \sum_{1\le i\le \frac{b-3}{3}}(q^{6i+2}+\dots+q^{d+2i}).
\end{equation}

We first consider the case $c=4$. For any given index $i=1,\dots, (b-3)/3$, the contribution of $i$ to the left side of (\ref{1}), which is given by
$$\sum_{j}(q^{6i+2j}+\dots+q^{6i+4j}),$$
clearly dominates the contribution of $i$ to the right side, since the largest degree that appears on the left side is
$$6i+4\lfloor b/2 -3i/2\rfloor \ge d+2i=(4b/3)+2i.$$
Also notice the following fact, which will be useful later on: When $b>6$ is even,  the coefficient of degree $2b-6$ on the left side of (\ref{1}), say $l_{2b-6}$, is \emph{strictly} greater than the corresponding coefficient on the right side, say $r_{2b-6}$. Indeed, a standard computation shows that $r_{2b-6}=2$. As for $l_{2b-6}$, one can see that it is at least 4, by considering the contribution, to the degree $2b-6$ coefficient on the left side of (\ref{1}), of the following four pairs of indices $(i,j)$:
$$((b-6)/3,2), ((b-6)/3,3), ((b-9)/3,3), ((b-9)/3,4).$$
This completes the proof that
\begin{equation}\label{rl}
l_{2b-6}>r_{2b-6}.
\end{equation}

Now let $c>4$. We consider the two sets of indices
$$1\le i\le \lfloor b/6\rfloor {\ }{\ }{\ }  \text{and} {\ }{\ }{\ } \lfloor b/6\rfloor +1\le i\le (b-3)/3$$
separately. When $i\le \lfloor b/6\rfloor$, we have that
$$\sum_{j}1\cdot(q^{6i+2j}+\dots+q^{ci+2i+cj})\ge (q^{6i+2}+\dots+q^{d+2i}),$$
since
$$ci+2i+c\lfloor b/2-2i(c-1)/c \rfloor \ge d+2i$$
for any $i\le \lfloor b/6\rfloor$.

Fix now an index $i$, $\lfloor b/6\rfloor +1 \le i\le (b-3)/3$. We  want to bound the sum contributed by $i$ to the right side of (\ref{1}) with the sum contributed by $i-\lfloor b/6\rfloor$ to the left side, of course without employing terms already used for $i\le \lfloor b/6\rfloor$. 

A quick thought tells us that we are done whenever the following inequality is true:
\begin{equation}\label{2}
\sum_{j}(q+\dots+q^{(c-4)(i-\lfloor b/6\rfloor)})(q^{6(i-\lfloor b/6\rfloor)+2j}+\dots+q^{(c+2)(i-\lfloor b/6\rfloor)+cj})\ge (q^{6i+2}+\dots+q^{d+2i}).
\end{equation} 

Clearly, (\ref{2}) is verified if we show that for each $i=\lfloor b/6\rfloor +1, \dots, (b-3)/3$, the left side  begins in degree no larger, and ends in degree no smaller, than the right side.  Since $j\le \lfloor b/2-2i(c-1)/c \rfloor$, standard computations give us that this is indeed the case for any $c\ge 5$, if we assume $b\ge 18$ when $c=5$ (the theorem is immediate to check directly for $b<18$ when $c=5$). This shows (\ref{1}).

In order to complete the proof of the theorem, it remains to bound
$$(q^2+q^3+\dots+q^{d})+\sum_{1\le i\le \frac{b-3}{3}}q^{6i},$$
using the KOH contribution to $\binom{b+c}{b}_q$ provided by some new family of partitions of $b$. 

We consider the following partitions:
$$\mu^i=(\mu^i_1=2,\dots,\mu^i_i=2,\mu^i_{i+1}=1,\dots,\mu^i_{b-i}=1),$$
where $1\le i\le \lceil b/2\rceil -1$. By Lemma \ref{koh}, the KOH contribution of $\mu^i$ to $\binom{b+c}{b}_q$ is:
$$F_{\mu^i}(q)= q^{2i}\binom{(c+2)i-4i+1}{1}_q\binom{(c+2)(b-i)-2b+1}{1}_q$$$$=q^{2i}\binom{ci-2i+1}{1}_q\binom{bc-ci-2i+1}{1}_q.$$

We have
$$2i+(bc-ci-2i+1)=bc-ci+1>bc/2,$$
for all $1\le i\le \lceil b/2\rceil -1$ and all $c$. Thus,
$$(1-q)F_{\mu^i}(q)=(1-q)q^{2i}(1+q+\dots+q^{ci-2i})\cdot \frac{1-q^{bc-ci-2i+1}}{1-q}$$$$=(1+q+\dots+q^{ci-2i})(q^{2i}-q^{bc-ci+1})=q^{2i}+q^{2i+1}+\dots+q^{ci}.$$

Since $c(\lceil b/2\rceil -1)\ge d=bc/3$ for any $c\ge 4$, it follows that
$$\sum_{1\le i\le  \lceil b/2\rceil -1} (1-q)F_{\mu^i}(q)=\sum_{1\le i\le  \lceil b/2\rceil -1} q^{2i}+q^{2i+1}+\dots+q^{ci}  \ge q^2+q^3+\dots+q^{d}.$$

We now want to check when 
\begin{equation}\label{66}
\sum_{1\le i\le  \lceil b/2\rceil -1} q^{2i}+q^{2i+1}+\dots+q^{ci}
\end{equation}
also simultaneously dominates $\sum_{1\le i\le \frac{b-3}{3}}q^{6i}$; that is, when in (\ref{66}) the coefficients in each degree $6, 12, \dots, 2b-6$ are at least 2.

Notice that every even power of $q$ that appears in (\ref{66}) for a given index $i< (b-3)/3$ also appears for $i+1$, with the only exception of $q^{2i}$. This, however, appears for $i-1$ if $i\ge 2$, so the total number of times it is present in (\ref{66}) is again at least 2. We conclude that a coefficient equal to 1 in some degree $n\equiv 0$ (mod 6), if it exists, can only be contributed by $i=(b-3)/3$, and such a degree $n$ must be in the range:
$$c(\lceil b/2\rceil -2)+1\le n\le c(\lceil b/2\rceil -1).$$

It follows that the theorem is proven whenever $c(\lceil b/2\rceil -2)+1>2b-6$. This is immediately verified to be the case for any $c\ge 6$. When $c=5$, the only exception is $b=n=6$, but for $b=6$ the theorem can easily be checked directly. Finally, when $c=4$, the only exceptions are  $b$  even and $n=2b-6$. If $b=6$ the theorem is again verified directly; if $b>6$, we employ inequality (\ref{rl}). The proof of the theorem is complete.
\end{proof}

\section{Acknowledgements} I am indebted to Richard Stanley for informing me of Bergeron's positivity problem, and for several helpful discussions during his visit to Michigan Tech in October 2016. In particular, it was during one of those discussions that we came up with the unimodality Conjecture \ref{conj}. I also wish to thank Fran\c{c}ois Bergeron for several comments, and Abdelmalek Abdesselam for informing me of his joint works \cite{AC1,AC2}. A final version of this note was drafted during a visiting professorship at MIT in Fall 2017, for which I am (once again) grateful to Richard. This work was done while I was partially supported by a Simons Foundation grant (\#274577).

%%%%%%%%%%%%%%%%%%%%%%%%%%%%%%%%%%%%%%%%%%%%%%%%%%%%%%%%%%%%%%

\end{document}